\documentclass[12pt]{article}
\usepackage{epsfig,amssymb}
\usepackage[english]{babel}
\usepackage{exscale}
\usepackage{indentfirst}
\usepackage{a4wide}
\usepackage[latin1]{inputenc}

\usepackage{amsfonts}
\usepackage{amsmath}

\usepackage{amsthm}

\newtheorem{theorem}{Theorem}[section]

\newtheorem{lemma}[theorem]{Lemma}

\usepackage{graphicx} % Required for inserting images

\title{New Upper Bounds for the Classical Ramsey Numbers $R(4,4,4)$, $R(3,4,5)$ and $R(3,3,6)$}

\author{Luis Boza}

\date{}

\begin{document}
\maketitle
\vspace{-10mm}
{\small Departamento de Matem\'atica Aplicada I,
Universidad de Sevilla, Sevilla, Spain, email:boza@us.es}

\begin{abstract}
The inequality
\[
R(k_1,\ldots,k_r)\le 2-r+\sum_{i=1}^r R(k_1,\ldots,k_{i-1},k_i-1,k_{i+1},\ldots,k_r)
\]
is well known, and it is strict whenever the right-hand side and at least one of the terms in the sum are even. Except for two known cases, the best upper bounds for classical Ramsey numbers with at least three colors have so far been obtained from this inequality.

In this paper we present new bounds such as $R(4,4,4)\le 229$, $R(3,4,5)\le 157$ and $R(3,3,6)\le 91$.
\end{abstract}

\section{Introduction}
Let $k_1,\ldots,k_r\ge 2$ be integers. The classical multicolor Ramsey number $R(k_1,\ldots,k_r)$ is defined as the smallest integer $N$ such that for every $r$-coloring $G$ of the edges of $K_N$ there exists $i$ such that $G$ contains a subgraph isomorphic to $K_{k_i}$ whose edges all have color $i$.
If $k_1=\ldots=k_r=k$, it is also denoted by $R_r(k):=R(\underbrace{k,\ldots,k}_r)$.
Clearly, permuting the $k_i$ does not change the value of the Ramsey number and $R(2,k_2,\ldots,k_r)=R(k_2,\ldots,k_r)$.

Few exact values are known, among them $R_2(4)=18$ \cite{GG}, $R(3,5)=14$ \cite{GG}, $R(3,6)=18$ \cite{Kery}, $R(4,5)=25$ \cite{MR4} and $R(3,3,4)=30$ \cite{CodFIM}.

Let
\[
P(k_1,\ldots,k_r)=2-r+\sum_{i=1}^r R(k_1,\ldots,k_{i-1},k_i-1,k_{i+1},\ldots,k_r).
\]
If $k_1=\ldots=k_r$, we also write $P_r(k_1)$.
The following general upper bound, implicit in \cite{GG}, is well known:
\begin{theorem} \label{or}
$R(k_1,\ldots,k_r)\le P(k_1,\ldots,k_r)$.
Moreover, the inequality is strict if the right-hand side is even and at least one of the terms in the sum is even.
\end{theorem}
By repeated application of Theorem \ref{or} one obtains the following upper bounds:
$R(3,4,4)\le 77$, $R(3,3,5)\le 57$, $R(3,4,5)\le 157$, $R_3(4)\le 230$, $R(3,3,3,4)\le 149$ and $R_7(3)\le 12861$ (see \cite{R} for further details).

For $r\ge 3$, there are only two known upper bounds for $R(k_1,\dots,k_r)$ that improve those obtained from Theorem \ref{or}, namely $R_4(3)\le 62$ \cite{FeKR} and $R(3,3,4)=30$.

If $G$ is an $r$-coloring and $v\in V(G)$, we denote by $d_i(v)$ the number of edges incident with $v$ that have color $i$. If $G$ has $n$ vertices, then clearly $\sum_{i=1}^r d_i(v)=n-1$.

We denote by $N_i(v)$ the set of the $d_i(v)$ vertices that are adjacent to $v$ by an edge of color $i$, and if $A\subseteq V(G)$ we write $G[A]$ for the $r$-coloring of the edges of the complete graph on vertex set $A$ induced by $G$.

We denote by ${\mathcal R}(k_1,\ldots,k_r;n)$ the set of $r$-colorings of the edges of $K_n$ such that for each $i$ there is no subgraph of $G$ isomorphic to $K_{k_i}$ whose edges are all of color $i$. Thus $R(k_1,\ldots,k_r)$ is the smallest $n$ such that ${\mathcal R}(k_1,\ldots,k_r;n)=\emptyset$.

If $G\in{\mathcal R}(k_1,\ldots,k_r;n)$, $v\in V(G)$ and $1\le i\le r$, then it is well known that
$G[N_i(v)]\in{\mathcal R}(k_1,\ldots,k_{i-1},k_i-1,k_{i+1},\ldots,k_r;d_i(v))$,
since if $G[N_i(v)]$ contained a subgraph $H$ isomorphic to $K_{k_i-1}$ with all edges of color $i$, then the vertices of $H$ together with $v$ would induce in $G$ a copy of $K_{k_i}$ whose edges all have color $i$.

As an easy consequence we have:
\begin{lemma} \label{ti}
Let $G\in{\mathcal R}(k_1,\ldots,k_r;n)$, let $v\in V(G)$, and let $1\le i_0\le r$. Then:
\begin{itemize}
\item $d_{i_0}(v)\le R(k_1,\ldots,k_{i_0-1},k_{i_0}-1,k_{i_0+1},\ldots,k_r)-1$.
\item If $n=P(k_1,\ldots,k_r)-1$, then
$d_{i_0}(v)=R(k_1,\ldots,k_{i_0-1},k_{i_0}-1,k_{i_0+1},\ldots,k_r)-1$.
Moreover, the number of edges of $G$ with color $i_0$ is
\[
\frac{(P(k_1,\ldots,k_r)-1)\bigl(R(k_1,\ldots,k_{i_0-1},k_{i_0}-1,k_{i_0+1},\ldots,k_r)-1\bigr)}{2}.
\]
\end{itemize}
\end{lemma}
\begin{proof}
For every $i\in\{1,\ldots,r\}$ we have
\[
G[N_i(v)]\in{\mathcal R}(k_1,\ldots,k_{i-1},k_i-1,k_{i+1},\ldots,k_r;d_i(v)).
\]
Hence $d_i(v)\le R(k_1,\ldots,k_{i-1},k_i-1,k_{i+1},\ldots,k_r)-1$, and in particular
\[
d_{i_0}(v)\le R(k_1,\ldots,k_{i_0-1},k_{i_0}-1,k_{i_0+1},\ldots,k_r)-1.
\]
If $n=P(k_1,\ldots,k_r)-1$, then
\begin{align*}
P(k_1,\ldots,k_r)-2=\sum_{i=1}^r d_i(v)
\le d_{i_0}(v)+\sum_{\substack{i=1\\ i\ne i_0}}^r (R(k_1,\ldots,k_{i-1},k_i-1,k_{i+1},\ldots,k_r)-1)\\
=d_{i_0}(v)+1-r+\sum_{\substack{i=1\\ i\ne i_0}}^r R(k_1,\ldots,k_{i-1},k_i-1,k_{i+1},\ldots,k_r), \text{ so }\\
d_{i_0}(v)\ge (P(k_1,\ldots,k_r)-2)-\left(1-r+\sum_{\substack{i=1\\ i\ne i_0}}^r R(k_1,\ldots,k_{i-1},k_i-1,k_{i+1},\ldots,k_r)\right)\\
=R(k_1,\ldots,k_{i_0-1},k_{i_0}-1,k_{i_0+1},\ldots,k_r)-1.
\end{align*}

Since $G$ is regular of degree $R(k_1,\ldots,k_{i_0-1},k_{i_0}-1,k_{i_0+1},\ldots,k_r)-1$ in color $i_0$ and has $P(k_1,\ldots,k_r)-1$ vertices, the number of edges of color $i_0$ is
\[
\frac{(P(k_1,\ldots,k_r)-1)\bigl(R(k_1,\ldots,k_{i_0-1},k_{i_0}-1,k_{i_0+1},\ldots,k_r)-1\bigr)}{2}.
\]
\end{proof}

Since
$(P(k_1,\ldots,k_r)-1)\bigl(R(k_1,\ldots,k_{i_0-1},k_{i_0}-1,k_{i_0+1},\ldots,k_r)-1\bigr)/2$
must be an integer, $P(k_1,\ldots,k_r)$ and $R(k_1,\ldots,k_{i_0-1},k_{i_0}-1,k_{i_0+1},\ldots,k_r)$ cannot be simultaneously even, which yields Theorem \ref{or}.

In this paper we obtain new upper bounds for classical Ramsey numbers with $r\ge 3$ that are not derived from Theorem \ref{or}. Among them are $R_3(4)\le 229$, $R(3,4,5)\le 157$ and $R(3,3,6)\le 91$.

\section{Main results}

The next result provides a criterion that allows one to improve the general bound of Theorem \ref{or} in certain modular cases.

\begin{theorem}\label{lem}
Assume that
\begin{enumerate}
\item $P(k_1,\ldots,k_r)\not\equiv 1\pmod{3}$;
\item there exists $i_0\in\{1,\ldots,r\}$ with $k_{i_0}\ge 4$ such that
\[
R(k_1,\ldots,k_{i_0-1},k_{i_0}-1,k_{i_0+1},\ldots,k_r)
=
P(k_1,\ldots,k_{i_0-1},k_{i_0}-1,k_{i_0+1},\ldots,k_r)\not\equiv 1\!\!\!\!\pmod{3};
\]
\item and
$R(k_1,\ldots,k_{i_0-1},k_{i_0}-2,k_{i_0+1},\ldots,k_r)\not\equiv 1\pmod{3}$.
\end{enumerate}
Then
\[
R(k_1,\ldots,k_r)\le P(k_1,\ldots,k_r)-1.
\]
\end{theorem}
\begin{proof}
Assume that $R(k_1,\ldots,k_r)=P(k_1,\ldots,k_r)$. Let
$G\in{\mathcal R}(k_1,\ldots,k_r;P(k_1,\ldots,k_r)-1)$ and let $v\in V(G)$. Since
$G[N_{i_0}(v)]\in{\mathcal R}(k_1,\ldots,k_{i_0-1},k_{i_0}-1,k_{i_0+1},\ldots,k_r;d_{i_0}(v))$
and, by Lemma \ref{ti}, $G[N_{i_0}(v)]$ has
\[
M:=\frac{\bigl(P(k_1,\ldots,k_{i_0-1},k_{i_0}-1,k_{i_0+1},\ldots,k_r)-1\bigr)
\bigl(R(k_1,\ldots,k_{i_0-1},k_{i_0}-2,k_{i_0+1},\ldots,k_r)-1\bigr)}{2}\]
edges of color $i_0$, it follows that $v$ lies in $M$ triangles of color $i_0$.
Therefore, summing over all vertices yields $(P(k_1,\ldots,k_r)-1)M$.
Each triangle is counted three times, so this sum must be a multiple of $3$.

By the hypotheses, neither $P(k_1,\ldots,k_r)-1$ nor any of the two factors defining $M$
is a multiple of $3$, which is a contradiction. Hence the result follows.
\end{proof}

As applications of Theorem \ref{lem} we obtain the following improvements:

\begin{theorem} \label{t1}
$R_3(4)\le 229.$
\end{theorem}
\begin{proof}
Let $r=3$, $k_1=4$, $k_2=4$, $k_3=4$ and $i_0=1$. If $P_3(4)\le 229$, then the result follows from Theorem \ref{or}. We may assume that $P_3(4)\ge 230$. Since $R_2(4)=18\not\equiv 1 \pmod{3}$, $P(3,4,4)=2-r+R(4,4)+2R(3,3,4)=2-3+18+2\cdot 30=77$,
while $230\le P_3(4)=2-3+3R(3,4,4)$, it follows that $R(3,4,4)\ge 77$, hence
$R(3,4,4)=P(3,4,4)=77\not\equiv 1\pmod{3}$.

Therefore $P_3(4)=2-3+3\cdot 77=230\not\equiv 1\pmod{3}$, and Theorem \ref{lem} yields the claim.
\end{proof}

\begin{theorem}  \label{t2}
$R(3,4,5)\le 157$.
\end{theorem}
\begin{proof}
Let $r=3$, $k_1=3$, $k_2=4$, $k_3=5$ and $i_0=2$. If $P(3,4,5)\le 157$, then the result follows from Theorem \ref{or}. We may assume that $P(3,4,5)\ge 158$. Since $R(3,5)=14\not\equiv 1 \pmod{3}$,
$P(3,3,5)=2-r+2R(3,5)+R(3,3,4)=2-3+2\cdot 14+30=57$, and
\[
158\le P(3,4,5)=2-3+R(4,5)+R(3,3,5)+R(3,4,4)\le 2-3+25+R(3,3,5)+77,
\]
we have $R(3,3,5)\ge 57$, hence $R(3,3,5)=P(3,3,5)=57\not\equiv 1\pmod{3}$.

Moreover, $R(3,4,4)\ge 158-(2-3+25+57)=77$, so $R(3,4,4)=77$. Thus $P(3,4,5)=2-3+25+57+77=158\not\equiv 1\pmod{3}$, and Theorem \ref{lem} gives the result.
\end{proof}

\begin{theorem}  \label{t3}
$R(3,3,6)\le 91$.
\end{theorem}
\begin{proof}
Let $r=3$, $k_1=3$, $k_2=3$, $k_3=6$ and $i_0=3$. If $P(3,3,6)\le 91$, then the result follows from Theorem \ref{or}. We may assume that $P(3,3,6)\ge 92$. Since $R(3,3,4)=30\not\equiv 1 \pmod{3}$, $P(3,3,5)=2-r+2R(3,5)+R(3,3,4)=2-3+2\cdot 14+30=57$,
and
\[
92\le P(3,3,6)=2-3+2R(3,6)+R(3,3,5)=2-3+2\cdot 18+R(3,3,5),
\]
we obtain $R(3,3,5)\ge 57$, hence $R(3,3,5)=P(3,3,5)=57\not\equiv 1\pmod{3}$.

Therefore $P(3,3,6)=2-3+2\cdot 18+57=92\not\equiv 1\pmod{3}$, and Theorem \ref{lem} yields the claim.
\end{proof}

\begin{theorem}  \label{t4}
$R(\underbrace{3,\ldots,3}_{10},4)\le 608152553$.
\end{theorem}
\begin{proof}
Let $r=11$, $k_1=\ldots=k_{10}=3$, $k_{11}=4$ and $i_0=11$. If $P(\underbrace{3,\ldots,3}_{10},4)\le 608152553$, then the result follows from Theorem \ref{or}. We may assume that $P(\underbrace{3,\ldots,3}_{10},4)\ge 608152554$.

From $R_7(3)\le 12861$, by repeated application of Theorem \ref{or} we obtain
$R_8(3)\le 102882$, $R_9(3)\le 925931$, $R_{10}(3)\le P_{10}(3)\le 9259302$ and $R_{11}(3)\le P_{11}(3)\le 101852313$.

From $R(3,3,3,4)\le 149$, by repeated application of Theorem \ref{or} we obtain
\begin{align*}
R(3,3,3,3,4)\le 900,\; R(\underbrace{3,\ldots,3}_{5},4)\le 6333, \; R(\underbrace{3,\ldots,3}_{6},4)\le 50854, \\
R(\underbrace{3,\ldots,3}_{7},4)\le 458853,\,
R(\underbrace{3,\ldots,3}_{8},4)\le 4596748,\,
R(\underbrace{3,\ldots,3}_{9},4)\le 50630025,\\ \text{and }
R(\underbrace{3,\ldots,3}_{10},4)\le
P(\underbrace{3,\ldots,3}_{10},4)\le 608152554.
\end{align*}
Since we are assuming $P(\underbrace{3,\ldots,3}_{10},4)\ge 608152554$, we deduce
\[
P(\underbrace{3,\ldots,3}_{10},4)=608152554\not\equiv 1\!\!\!\!\pmod{3}.
\]

If $R_{11}(3)\le 101852312$, then $
P(\underbrace{3,\ldots,3}_{10},4)=2-11+10R(\underbrace{3,\ldots,3}_9,4)+R_{11}(3)
\le 2-11+10\cdot 50630025+101852312=608152553$, which is a contradiction. Hence we may assume that $R_{11}(3)=P_{11}(3)=101852313\not\equiv 1\pmod{3}$.

If $R_{10}(3)\le 9259301$, then $P_{11}(3)=2-11+11R_{10}(3)\le 2-11+11\cdot 9259301=101852302$,
a contradiction. Therefore we may assume that $R_{10}(3)=9259302\not\equiv 1\pmod{3}$, and Theorem \ref{lem} yields the result.
\end{proof}

Finally, we note that any further improvement obtained from Theorem \ref{lem} is a consequence of Theorems \ref{t1}--\ref{t4}. In particular, for $r\le 10$ no improved upper bounds are obtained for $R(\underbrace{3,\ldots,3}_{r-1},5)$ and $R(\underbrace{3,\ldots,3}_{r-2},4,4)$ beyond those derived from Theorem \ref{or}.

\end{document}